\documentclass[11pt,twoside,reqno]{amsart}
\usepackage{times}
\usepackage{array} 
\usepackage{mathtools}
\usepackage{mathrsfs}    
\usepackage{amssymb,amsmath,amsfonts}
\usepackage{algorithm}
\usepackage{multirow}
\usepackage{xcolor}

\usepackage[colorlinks = true,
linkcolor = blue,
urlcolor  = violet, 
citecolor = blue,
anchorcolor = blue]{hyperref}

\newtheorem{thm}{Theorem}[section]
\newtheorem{prop}[thm]{Proposition}

\theoremstyle{definition}
\newtheorem{definition}[thm]{Definition}

\theoremstyle{remark}
\newtheorem{remark}[thm]{Remark}

\numberwithin{equation}{section} 

\makeatletter
\newcommand{\thickhline}{%
	\noalign {\ifnum 0=`}\fi \hrule height 1pt
	\futurelet \reserved@a \@xhline
}
\usepackage{enumerate,cite}
\newcommand{\conv}{\mbox{\rm conv}\,}

\newcommand{\cl}{\mbox{\rm cl}\,}

\newcommand{\dom}{{\rm dom}\,}

\newcommand{\gph}{\mbox{\rm gph}\,}
\def\R{\mathbb{R}}
\def\N{\mathbb{N}}

\newcommand{\Limsup}{\mathop{{\rm Lim}\,{\rm sup}}}

\begin{document}
\title[Second-order KKT optimality conditions for $C^{1,1}$ VOPs]{On second-order Karush--Kuhn--Tucker optimality conditions for $C^{1,1}$ vector optimization problems}
\author[N.V. Tuyen]{Nguyen Van Tuyen$^{1,*}$}
\maketitle
\vspace*{-0.6cm}
\begin{center}
{\footnotesize {\it
$^1$Department of Mathematics, Hanoi Pedagogical University 2, Xuan Hoa, Phuc Yen, Vinh Phuc, Vietnam
}}\end{center} 

\begin{abstract} This paper focuses on optimality conditions for $C^{1,1}$ vector optimization problems with inequality constraints. By employing the limiting second-order subdifferential and the second-order tangent set, we introduce a new type of  second-order constraint qualification in the sense of Abadie. Then we establish some second-order  necessary optimality conditions of Karush--Kuhn--Tucker-type for local (weak) efficient solutions of the considered problem.  In addition, we provide some sufficient conditions for a local efficient solution of the such problem. The obtained results improve existing ones in the literature.
\end{abstract}

\maketitle


\renewcommand{\thefootnote}{}
\footnotetext{ $^*$Corresponding author.
\par
E-mail addresses: nguyenvantuyen83@hpu2.edu.vn; tuyensp2@yahoo.com (N.V. Tuyen) 
}

%%%%%%%%%%%%%%%%%%%%%%%%%%%%%%%%%%%%%%%%%%%%%%%%%%%%%%%%%%%%%%%%%%%%%%
\section{Introduction}
%%%%%%%%%%%%%%%%%%%%%%%%%%%%%%%%%%%%%%%%%%%%%%%%%%%%%%%%%%%%%%%%%%%%%%
The investigation of optimality conditions is one of the most attractive topics in optimization theory.  It is well-known that, without the convexity,  first-order  optimality conditions (Fritz John/Karush--Kuhn--Tucker type) are usually not sufficient ones. This motivated mathematicians to study second-order optimality conditions.  The second-order optimality conditions complement first-order ones in eliminating non-optimal solutions. For $C^2$ (twice continuously differentiable) constrained optimization problems, it is well-known that  the positive definiteness of the Hessian of the associated Lagrangian function is a sufficient condition for the optimality; see, for example,  \cite{Bonnans2000}. For non-$C^2$  problems, to obtain the second-order optimality conditions, many different kinds of generalized second-order  derivatives have been proposed; see, for example, \cite{Dhara12,Guerraggio01,Guerragio011,novo10,hien84,Huy16,Jeyakumar08,Maeda04,Mor92,Mor06a,Mor12,Mor13,Mor132,Mor133,Mor142,Mor143,Mor15,Mor152,Mordukhovich_2018,Mordukhovich-2024,Tuan15,Tuan16}.  

In the literature, there are two generally independent approaches  dealing with generalized second-order differentiations. The first one is based on the Taylor expansion, while the other is defined by induction, i.e., the second-order derivative of a real-valued function is the derivative of its first-order one. In \cite{Mor92}, Mordukhovich proposed a new approach to construct second-order subdifferentials of extended-real-valued functions as the  coderivative  of the subgradient mapping.   The    second-order subdifferential theory, as introduced by Mordukhovich,  and its modification  were successfully employed in the study of a broad spectrum of other important issues in variational
analysis and its applications; see, for example, \cite{Mor92,Mor94,Mor06a,Levy00,Mordukhovich_2018,Mordukhovich-2024,Poliquin98}. We refer the reader to the recent book by Mordukhovich~\cite{Mordukhovich-2024}. This comprehensive work, consisting of nine  chapters, provides a valuable reference for recent researchers in this area. 

In \cite{Huy16},  Huy and Tuyen introduced the concept of second-order symmetric subdifferential and developed its calculus rules. By using the second-order symmetric subdifferential, the second-order tangent set and the asymptotic second-order tangent cone, they established some second-order necessary and sufficient optimality conditions for optimization problems with geometric constraints. As shown in \cite{Huy16}, the second-order symmetric subdifferential may be strictly smaller than the Clarke  subdifferential, and  has some nice properties. In particular, every $C^{1,1}$ function has Taylor expansion in terms of its second-order symmetric subdifferential. Thereafter, in \cite{Huy162,Tuyen-Huy-Kim}, the authors used second-order symmetric subdifferentials to derive second-order optimality conditions of Karush--Kuhn--Tucker (KKT) type for $C^{1,1}$ vector optimization problems with inequality constraints. Then, in \cite{Feng_Li_2020}, Feng and Li introduced a  Taylor formula in the form of inequality for limiting second-order subdifferentials  and obtained some second-order Fritz John type optimality
conditions for $C^{1,1}$ scalar optimization problems with inequality constraints.  Recently, An and Tuyen \cite{An-Tuyen-24} derived some optimality conditions for $C^{1,1}$  optimization problems subject to inequality and equality constraints by employing the concept of limiting (Mordukhovich) second-order subdifferentials to the Lagrangian function associated with the considered problem.

The aim of this work to extend and improve results in \cite{Feng_Li_2020,Huy162,Tuyen-Huy-Kim}  to $C^{1,1}$ vector optimizations problems. To do this, we first introduce a new type of second-order constraint qualification in the sense of Abadie and some sufficient conditions for this constraint qualification. Under the Abadie second-order constraint qualification, we obtain second-order KKT necessary optimality conditions for efficiency of the considered problem. We also derive a second-order sufficient optimality condition of strong KKT-type for local efficient solutions.

The rest of this paper is organized as
follows.  In Section \ref{section2}, we recall some definitions and preliminary results from variational analysis and generalized differentiation. Section \ref{Abadie_SORC} presents main results. Section \ref{Discussion} draws some conclusions.   
%====================================================
\section{Preliminaries} \label{section2}
Throughout the paper, the considered spaces are finite-dimensional Euclidean with the inner product and the norm being denoted by $\langle \cdot, \cdot \rangle$ and by $\|\cdot\|$, respectively. 

For $a, b\in\R^m$, by $a\leqq b$, we mean $a_l\leq b_l$ for all $l=1, \ldots, m$; by $a\leq b$, we mean $a\leqq b$ and $a\neq b$; and by $a<b$, we mean $a_l<b_l$ for all $l=1, \ldots, m$.
  
Let $\Omega$ be a nonempty subset in $\R^n$. The  {\it closure} and {\it convex hull}  of $\Omega$ are denoted, respectively, by $\mbox{cl}\,\Omega$ and  $\conv\,\Omega$. The unit sphere in $\R^n$ is denoted by $\mathbb{S}^n$. We denote the nonnegative orthant in $\R^n$ by $\R^n_+$.  
 
 Let  $F : \R^n \rightrightarrows \R^m$ be a set-valued mapping. The \textit{domain} and the \textit{graph}  of $F$ are given, respectively, by
 $${\rm dom}\,F=\{x\in \R^n : F(x)\not= \emptyset\} $$
 and
 $$ {\rm gph}\,F=\{(x,y)\in \R^n \times \R^m : y \in F(x)\}.$$
 The set-valued mapping $F$ is called   \textit{proper} if $\dom F \not= \emptyset.$ The \textit{Painlev\'e-Kuratowski outer/upper limit} of $F$ at $\bar x$ is defined by
 \begin{align*} 
 	\Limsup\limits_{x\rightarrow \bar x} F(x):=\bigg\{ y\in \mathbb{R}^m : \exists x_k \rightarrow \bar x, y_k \rightarrow y \ \mbox{with}\ y_k\in F(x_k), \forall k=1,2,....\bigg\}.
 \end{align*} 
 \begin{definition}{\rm Let $\Omega$ be a nonempty subset in $\R^n$,   $\bar x\in \Omega$, and $u\in\R^n$. 
\begin{enumerate}[\rm(i)]
\item The {\em tangent cone} to $\Omega$ at $\bar x\in \Omega$ is defined by 
 		$$T(\Omega; \bar x):=\{d\in\R^n\,:\,\exists t_k\downarrow 0, \exists d^k\to d, \bar x+t_kd^k\in \Omega, \ \ \forall k\in \N\}.$$	
\item The {\em second-order tangent set} to $\Omega$ at $\bar x$ with respect to the direction $u$ is defined by 
 		$$T^2(\Omega; \bar x, u):=\left\{v\in\R^n:\exists t_k\downarrow 0, \exists v^k\to v, \bar x+t_ku+\frac12t_k^2v^k\in \Omega,\, \forall k\in \N\right\}.$$
\end{enumerate}		
 	}	
 \end{definition}
 By definition, $T(\,\cdot\,; \bar x)$ and $T^2(\,\cdot\,; \bar x, u)$  are isotone, i.e., if $\Omega^1\subset \Omega^2$, then
 \begin{align*}
 T(\Omega^1; \bar x)\subset T(\Omega^2; \bar x)\ \ \text{and}\ \ T^2(\Omega^1; \bar x, u)\subset  T^2(\Omega^2; \bar x, u).
 \end{align*}

It is well-known that $T(\Omega; \bar x)$ is a nonempty closed cone, $T^2(\Omega; \bar x, u)$ is closed, and $T^2(\Omega; \bar x, u)=\emptyset$ if $u\notin T(\Omega; \bar x)$. We refer the reader to \cite{Giorgi et al,Khan-et al} and the bibliography therein for other interesting properties of the above tangent sets.

\begin{definition}[{see~\cite{Mor06a}}] \rm  
	Let $\Omega$ be a nonempty subset of $\mathbb{R}^n$ and $\bar x \in \Omega$. The \textit{Fr\'echet/regular normal cone}  to $\Omega$ at $\bar x$ is defined by
	\begin{align*}
		\widehat N(\bar x, \Omega)=\Big\{ v\in \mathbb{R}^n : \limsup\limits_{x \xrightarrow{\Omega}\bar x} \dfrac{\langle v, x-\bar x \rangle}{\|x-\bar x\|} \leq 0 \Big\},
	\end{align*}
	where $x \xrightarrow{\Omega} \bar x$ means that $x \rightarrow \bar x$ and $ x\in \Omega$. The \textit{limiting/Mordukhovich normal cone} to $\Omega$ at $\bar x$ is given by
	\begin{align*}
		N(\bar x, \Omega)=\Limsup\limits_{ x \xrightarrow{\Omega} \bar x} \widehat{N}(x, \Omega).
	\end{align*}
	We put $\widehat N(\bar x,\Omega)=N(\bar x,\Omega): =\emptyset$ if $\bar x \not\in \Omega$.
\end{definition}	

By definition, one has $\widehat N(\bar x,\Omega) \subset N(\bar x,\Omega)$ and when $\Omega$ is convex, then the regular normals to $\Omega$  at $\bar x$ coincides with the limiting normal cone and both constructions reduce to the normal cone in the sense of convex analysis, i.e.,
\begin{align*}
	\widehat N(\bar x,\Omega) = N(\bar x,\Omega):=\{ v\in \mathbb{R}^n : \langle v,x- \bar x \rangle \le 0,  \ \forall x\in \Omega\}.
\end{align*}
 
Consider an {\it extended-real-valued function}   $\varphi\colon \R^n \to \overline \R:=\R\cup\{+\infty\}$. The {\it  epigraph}, {\it  hypergraph} and  {\it domain} of $\varphi$ are denoted, respectively, by
 \begin{align*}
 	\mbox{epi }\varphi&:=\{(x, \alpha)\in\R^n\times\R \,:\,  \alpha\geq \varphi(x) \},
 	\\
 	\mbox{hypo }\varphi &:=\{(x, \alpha)\in\R^n\times\R \,:\, \alpha\leq \varphi(x) \}, 
 	\\
 	\mbox{dom }\varphi &:= \{x\in \R^n \,:\, \varphi(x)<+\infty \}.
 \end{align*}
 The function $\varphi$ is called {\em proper} if $\dom\varphi$ is nonempty.

 \begin{definition}[{see \cite{Mor06a}}]{\rm 
 		Given $\bar x\in \mbox{dom }\varphi.$ The sets
 		\begin{align*}
 			\partial \varphi (\bar x)&:=\{x^*\in \R^n \,:\, (x^*, -1)\in N((\bar x, \varphi (\bar x)); \mbox{epi }\varphi )\}
 			\\
 			\partial ^+ \varphi (\bar x)&:=\{x^*\in \R^n \,:\, (-x^*, 1)\in N((\bar x, \varphi (\bar x)); \mbox{hypo }\varphi )\},
 			\\
 			\partial _S \varphi (\bar x)&:=\partial \varphi (\bar x)\cup \partial^+ \varphi (\bar x), 
 			\\
 			\partial _C \varphi (\bar x)&:=\cl\conv \partial \varphi (\bar x)
 		\end{align*}
 		are called the {\it limiting/Mordukhovich subdifferential}, the {\em upper subdifferential}, the {\em symmetric subdifferential}, and the {\em Clarke subdifferential} of $\varphi$ at $\bar x$, respectively. If $\bar x\notin\dom\varphi$, then we put 
 		$$\partial \varphi (\bar x)=\partial ^+ \varphi (\bar x)=\partial_S \varphi (\bar x)=\partial_C \varphi (\bar x):=\emptyset.$$
 			}
 \end{definition}
In contrast with the Clarke subdifferential, the limiting (symmetric) subdifferential may be nonconvex and, by definition, it is clear that
\begin{equation}\label{equa-new-1}
\partial \varphi (\bar x) \subseteq \partial _S \varphi (\bar x) \subseteq \partial _C \varphi (\bar x),
\end{equation}
and both inclusions may be strict; see \cite[pp. 92--93]{Mor06a}. 

\begin{definition}[{see~\cite[Definition 1.11]{Mordukhovich_2018}}]  \rm   Let $  F: \mathbb{R}^n\rightrightarrows{\mathbb{R}^m}$ be a set-valued mapping and  $(\bar x, \bar y) \in  {\rm{gph}}\, F$.	The  \textit{limiting/Mordukhovich coderivative}  of  $F$ at $(\bar x, \bar y)$ is a multifunction $D^* F(\bar x, \bar y): \mathbb{R}^m \rightrightarrows{\mathbb{R}^n}$ with the values
	\begin{equation}\label{equa-new-2}
		D^* F(\bar x,\bar y)(u):= \left\{v\in \mathbb{R}^n : (v, -u) \in  N\left((\bar x, \bar y),\gph\, F\right)\right\}, \ \ u \in \mathbb{R}^m.
	\end{equation}
	If $(\bar x, \bar y) \notin {\rm{gph}}\, F$, we  put $ D^* F(\bar x, \bar y)(u):=\emptyset$ for any $u\in \mathbb{R}^m$. When $F$ is single-valued at $\bar x$ with $\bar y=F(\bar x)$,  the symbol $\bar y$ in the notation  $D^* F(\bar x,\bar y)$ will be omitted. 
\end{definition}	 
If the limiting normal cone in \eqref{equa-new-2} is replaced by Clarke normal one, then the set 
$$D_C^* F(\bar x,\bar y)(u):= \left\{v\in \mathbb{R}^n : (v, -u) \in  N_C\left((\bar x, \bar y),\gph\, F\right)\right\}, \ \ u \in \mathbb{R}^m$$
is called the {\em Clarke coderivative} of $F$ at $(\bar x, \bar y)$ with respect to $v$.

We now recall the definition of the limiting second-order subdifferential. This is first introduced by Mordukhovich in \cite{Mor92}.
 \begin{definition}\label{dn:11n}{\rm
 		Let $(\bar x, \bar y)\in \mbox{gph }\partial\varphi$. The {\it limiting/Mordukhovich second-order subdifferential}  of $\varphi$ at $\bar x$ relative to $\bar y$ is a set-valued mapping $\partial^2\varphi (\bar x, \bar y)\colon \R^n \rightrightarrows \R^n$ defined by
 		\begin{equation*}\label{eq:11n}
 			\partial^2\varphi (\bar x, \bar y)(u):=(D^*\partial\varphi)(\bar x, \bar y)(u)=\{v \,:\, (v, -u)\in N(((\bar x, \bar y)); \mbox{gph} \partial \varphi )\}, \ \ u\in\R^n.
 		\end{equation*}
 	}
 \end{definition}  
Note that if $\varphi$ is strictly differentiable at $\bar x$, then $\partial\varphi(\bar x)=\{\nabla\varphi(\bar x)\}$ with $\nabla\varphi(\bar x)$ being the Fr\'echet derivative of $\varphi$ at $\bar x$, see \cite[Corollary 1.82]{Mor06a}. Recall that the function $\varphi : \R^n\to\R^m$ is said to be {\em strictly differentiable} at $\bar x$ if and only if there is a linear continuous operator $\nabla \varphi(\bar x) : \R^n\to \R^m$, called the {\em Fr\'echet derivative} of $\varphi$ at $\bar x$, such that
\begin{equation*}
	\lim_{\substack{x\to\bar x\\ u\to\bar x}}\dfrac{\varphi(x)-\varphi(u)-\langle\nabla \varphi(\bar x), x-u\rangle}{\|x-u\|}=0.
\end{equation*} 
Clearly, if $\varphi\in C^{1,1}(\R^n)$, then $\varphi$ is strictly differentiable on $\R^n$ and so $\partial^2\varphi (\bar x, \bar y)(u)=(D^*\nabla\varphi)(\bar x)(u)$. We recall here that a real-valued function  is said to be a $C^{1,1}$ function if it is Fr\'echet differentiable with a locally Lipschitz gradient.

In Definition \ref{dn:11n}, if the limiting coderivative is replaced by the Clarke coderivative, then we obtain the corresponding {\em  Clarke second-order subdifferential} $\partial_C^2\varphi (\bar x, \bar y)$. 
\begin{prop}[{see \cite[Theorem 1.90]{Mor06a}}]
If $\varphi\in C^{1, 1}(\mathbb{R}^n)$, then one has 
\begin{equation*}\label{eq:14n}
	\partial^2 \varphi (\bar x)(u):= \partial^2\varphi (\bar x, \nabla \varphi(\bar x))(u)=(D^*\nabla\varphi)(\bar x)(u)=\partial\langle u, \nabla \varphi \rangle (\bar x) \ \ \forall u, \bar x\in \mathbb{R}^n.
\end{equation*} 
\end{prop}

In \cite{Huy16}, the authors introduced the so-called the second-order symmetric subdifferential in the sense of Mordukhovich as follows. 
 
\begin{definition}[{see \cite[Definition 2.6]{Huy16}}]{\rm 		Let   $\varphi\in C^{1, 1}(\R^n)$ and $\bar x\in \R^n$. The {\em second-order  symmetric subdifferential} of $\varphi$ at $\bar x$ is a multifunction $\partial_S^2\varphi(\bar x)\colon \R^n \rightrightarrows \R^n$  defined by
 		\begin{equation*}
 			\partial_S^2\varphi (\bar x)(u):=\partial_S\langle u, \nabla \varphi \rangle (\bar x), \ \ \forall u\in \R^n.
 		\end{equation*}
 	}
 \end{definition}
By definition and \eqref{equa-new-1}, one has
$$\partial^2\varphi (\bar x)(u)\subset\partial_S^2\varphi (\bar x)(u) \subset \partial_C^2\varphi (\bar x, \bar y)(u)$$
and the above inclusions may be strict.

We end this section by recall some results on the properties of second-order subdifferentials
that will be needed in the sequel.
\begin{prop}[{see \cite{Huy16,Mor06a,Mordukhovich_2018}}]\label{property}
	Let $\varphi \in C^{1,1}(\R^n)$ and $\bar x\in \R^n$. The following assertions hold:
\begin{enumerate}[\rm (i)]
	\item For any $\lambda \ge 0$, one has $\partial^2 \varphi(\bar x)(\lambda u) = \lambda \partial^2 \varphi(\bar x)(u), \forall u\in \mathbb{R}^n.$
	%\item For any $\lambda \in\R$, one has $\partial_S^2 \varphi(\bar x)(\lambda u) = \lambda \partial_S^2 \varphi(\bar x)(u), \forall u\in \mathbb{R}^n.$
	\item For any $u\in \mathbb{R}^n$, the set  $\partial^2 \varphi(\bar x) (u)$ is nonempty and compact.
	\item For any $u\in \mathbb{R}^n$, the mapping $ \partial^2 \varphi(\cdot) (u)$  is locally bounded around $\bar x$ and if $x^k \to \bar x$, $v^k \to v$, where $v^k \in \partial^2 \varphi(x^k)(u)$  for all $k\in \mathbb{N}$, then $v\in \partial^2 \varphi(\bar x)(u)$. 
\end{enumerate} 
\end{prop}

The Taylor formula in  the  form of inequalities for $C^{1,1}$ functions, employing the limiting second-order subdifferential, plays an important role for our research. 
\begin{thm} [{see~\cite[Theorem 3.1]{Feng_Li_2020}}]\label{Taylor_formula}
	Let $\varphi$ be of class $C^{1,1}(\R^n)$ and $a,b\in\R^n$. Then, there exist $z\in\partial^2\varphi(\xi) (b-a)$, where $\xi \in [a,b]$, $z'\in \partial^2 \varphi(\xi') (b-a)$, where $\xi' \in [a,b]$, such~that
	$$ \dfrac{1}{2} \langle z', b-a \rangle \le \varphi(b)-\varphi(a)-\langle\nabla \varphi(a), b-a \rangle \le \dfrac{1}{2}\langle z, b-a \rangle.$$
\end{thm}
 
 %===========================================================================
 
 \section{Main results}
 \label{Abadie_SORC}
  In this paper, we investigate the following constrained vector optimization  problem
 \begin{align*}
 	& \text{min}_{\R^m_+}\, f(x)\label{problem} \tag{VOP}
 	\\
 	&\text{s. t.}\ \ x\in X:=\{x\in\mathbb{R}^n\,:\, g(x)\leqq 0\},
 \end{align*}
 where $f:=(f_l)$, $l\in L:=\{1, \ldots, m\}$, and  $g:=(g_i)$, $i\in I:=\{1, \ldots, p\}$, are vector-valued functions with $C^{1,1}$ components defined on $\R^n$. 
 \subsection{Abadie second-order constraint qualification}
 In this subsection, we propose a type  of second-order constraint qualification in the sense of Abadie for problem  \eqref{problem}  and establish some conditions which assure  that this constraint qualification holds true. 
 
Fix any $\bar x\in X$ and $u\in\R^n$. Then by Proposition \ref{property}(ii), $\partial^2f_l(\bar x)(u)$, $l\in L$, and $\partial^2g_i(\bar x)(u)$, $i\in I$,  are nonempty and compact sets. Hence, there exist  $\xi^{*l}$ and  $\xi_*^{l}$ (resp., $\zeta^{*i}$ and  $\zeta_*^{i}$)  are elements  in  $\partial^2f_l(\bar x)(u)$  (resp., $\partial^2g_i(\bar x)(u)$)   such that    
 \begin{align*}
 	\langle\xi^{*l}, u\rangle&:=\max \left\{\langle \xi^l, u\rangle\,:\, \xi^l\in \partial^2f_l(\bar x)(u)\right\},\ \ l\in L,
 	\\
 	\langle\xi_*^{l}, u\rangle&:=\min \left\{\langle \xi^l, u\rangle\,:\, \xi^l\in \partial^2f_l(\bar x)(u)\right\},\ \ l\in L,
 	\\
 	\langle\zeta^{*i}, u\rangle&:=\max \left\{\langle \zeta^i, u\rangle\,:\,\zeta^i\in \partial^2g_i(\bar  x)(u) \right\},\ \ i\in I,
 	\\
 	\langle\zeta_*^{i}, u\rangle&:=\min \left\{\langle \zeta^i, u\rangle\,:\,\zeta^i\in \partial^2g_i(\bar x)(u) \right\},\ \ i\in I.
 \end{align*}
For any  $a=(a_1, a_2)$ and $b=(b_1, b_2)$ in $\R^2$, we denote the lexicographic order by
\begin{align*}
	a&\leqq_{\rm lex} b,\ \  {\rm if} \ \ a_1<b_1\ \   {\rm or} \ \  a_1=b_1\ \ {\rm and }\ \   a_2\leq b_2,
	\\
	a&<_{\rm lex} b,\ \  {\rm if} \ \ a_1<b_1\ \   {\rm or} \ \  a_1=b_1\ \ {\rm and }\ \   a_2< b_2.
\end{align*}
 For  $u, v \in\R^n$, put
 \begin{align*} 
 	F^2_l(u, v)&:=\left(\langle \nabla f_l(\bar x), u\rangle, \langle \nabla f_l(\bar x), v\rangle+\langle\xi^{*l}, u\rangle\right),\ \ l\in L\\
 	G^2_i(u, v)&:=\left(\langle \nabla g_i(\bar x), u\rangle, \langle \nabla g_i(\bar x), v\rangle+\langle\zeta^{*i}, u\rangle\right),\ \ i\in I,
 	\\
 	F^{2-}_l(u, v)&:=\left(\langle \nabla f_l(\bar x), u\rangle, \langle \nabla f_l(\bar x), v\rangle+\langle\xi_*^{l}, u\rangle\right),\ \ l\in L\\
 	G^{2-}_i(u, v)&:=\left(\langle \nabla g_i(\bar x), u\rangle, \langle \nabla g_i(\bar x), v\rangle+\langle\zeta_*^{i}, u\rangle\right),\ \ i\in I,
 \end{align*}
 \begin{equation*}
 	L^2(X; \bar x, u):=\{v\in\R^n\, :\, G^2_i(u, v)\leqq_{\rm lex} (0,0),\ \ i\in I(\bar x)\},
 \end{equation*}
and 
 \begin{equation*}
	L^{-2}(X; \bar x, u):=\{v\in\R^n\, :\, G^{-2}_i(u, v)\leqq_{\rm lex} (0,0),\ \ i\in I(\bar x)\},
\end{equation*}
where $I(\bar x)$ is the {\em  active index set} to $\bar x$ and defined by
$$I(\bar x):=\{i\in I : g_i(\bar x)=0\}.$$

By definition, it is clear that $L^2(X; \bar x, u)\subset L^{2^-}(X; \bar x, u)$. 
 
The following result gives an upper estimate of the second-order tangent set to the constraint set of problem \eqref{problem}.  
 \begin{prop}\label{lemma1}
 	Let $u\in\R^n$ be any vector.  Then the following inclusion  holds
 	$$T^2(X; \bar x, u)\subset  L^{2^-}(X; \bar x, u).$$
 \end{prop}
\begin{proof}
Fix any $v\in T^2(X; \bar x, u)$. Then, there exist sequences $t_k\downarrow 0$  and   $v^k$ converging to $v$ such that
 $$x^k:=\bar x+t_ku+\frac12t_k^2v^k\in X,\ \ \forall k\in\N.$$ 
Hence, for each $i\in I(\bar x)$, one has  $g_i(x^k)-g_i(\bar x)\leq 0$ for all $k\in\N$. By the mean value theorem for differentiable functions, there exists $\theta^k\in(\bar x, x^k)$ such that 
 $$\langle \nabla g_i(\theta^k), t_ku+\frac12t_k^2v^k\rangle\leq 0,\ \ \forall k\in\N.$$
Dividing two sides of the above inequality by $t_k$ and letting $k\to\infty$, we obtain $\langle \nabla g_i(\bar x), u\rangle\leq 0$. We claim that 
$$G^{2^-}_i(u, v)\leqq_{\rm lex} (0, 0),$$
 or, equivalently,
 \begin{equation}\label{equa:1}
 	(\langle \nabla g_i(\bar x), u\rangle, \langle \nabla g_i(\bar x), v\rangle+\langle\zeta_*^{i}, u\rangle)\leqq_{\rm lex} (0, 0).
 \end{equation}
 Clearly, \eqref{equa:1} is satisfied if $\langle \nabla g_i(\bar x), u\rangle<0$. When  $\langle \nabla g_i(\bar x), u\rangle =0$, then we have
 $$g_i(x^k)-g_i(\bar x)=[g_i(x^k)-g_i(\bar x+t_ku)]+[g_i(\bar x+t_ku)-g_i(\bar x)-t_k\langle\nabla g_i(\bar x), u\rangle].$$
 By the mean value theorem for differentiable functions, there exists $\gamma^k\in (\bar x+t_ku, x^k)$ such that
 \begin{equation}
 	g_i(x^k)-g_i(\bar x+t_ku)=\langle \nabla g_i(\gamma^k),  \frac12t_k^2v^k\rangle=\frac12t_k^2\langle \nabla g_i(\gamma^k),  v^k\rangle. \label{equ:11}
 \end{equation}
By Theorem \ref{Taylor_formula}, there exist $\sigma^k\in (\bar x, \bar x+t_ku)$ and $w^k\in \partial^2g_i(\sigma^k)(t_ku)$ such that
 $$g_i(\bar x+t_ku)-g_i(\bar x)-t_k\langle\nabla g_i(\bar x), u\rangle \geq \frac12 \langle w^k, t_ku\rangle= \frac12t_k \langle w^k, u\rangle.$$
Since $\partial^2g_i(\sigma^k)(t_ku)=t_k\partial^2g_i(\sigma^k)(u)$, one has $w^k=t_k\zeta^k$ for some $\zeta^k\in \partial^2g_i(\sigma^k)(u)$. Thus 
\begin{equation}\label{equa-new-5}
g_i(\bar x+t_ku)-g_i(\bar x)-t_k\langle\nabla g_i(\bar x), u\rangle \geq \frac12t^2_k \langle \zeta^k, u\rangle.
\end{equation}
This and \eqref{equ:11} imply  that
 \begin{equation}\label{equa-new-3} 
 0\geq g_i(x^k)-g_i(\bar x)\geq\frac12t^2_s[\langle \nabla g_i(\gamma^k),  v^k\rangle+\langle \zeta^k, u\rangle].
 \end{equation}
Hence,
 \begin{equation}\label{equ:13}
 \langle \nabla g_i(\gamma^k),  v^k\rangle+\langle \zeta^k, u\rangle\leq 0,\ \  \forall k\in\N. 
 \end{equation}
 
 Since $\partial^2 g_i(\cdot)$ is locally bounded around $\bar x$ and  $\lim_{k\to\infty}\sigma^k=\bar x$, the sequence $\zeta^k$ is bounded. Without loss of any generality, we may assume that $\zeta^k$ converges to $\zeta^i$. By Proposition \ref{property}(iii), $\zeta^i\in \partial^2g_i(\bar x)(u)$. Since $g_i\in C^{1,1}(\R^n)$, one has
 $$\lim_{k\to\infty}\langle\nabla g_i(\gamma^k),  v^k\rangle=\langle\nabla g_i(\bar x),  v\rangle.$$
 Letting $k\to\infty$ in \eqref{equ:13} we arrive at $\langle\nabla g_i(\bar x),  v\rangle+ \langle \zeta^i, u\rangle\leq 0$. Consequently,
 $$\langle\nabla g_i(\bar x),  v\rangle+ \langle \zeta_*^i, u\rangle\leq 0.$$
 This means that \eqref{equa:1} holds true and so $v\in L^{2^-}(X; \bar x, u)$. The proof is complete.
\end{proof}
 
 We now introduce a type  of second-order  constraint qualification in the sense of Abadie. 
 \begin{definition} \rm Let  $\bar x\in X$ and $u\in\R^n$. We say that $\bar x$ satisfies  the {\em  Abadie second-order  constraint qualification} with respect to the direction $u$ if 
 	\[
 	L^{2}(X; \bar x, u)\subset  T^2(X; \bar x, u). \tag{$ASCQ$}\label{ASOCQ}
 	\]
 \end{definition}
\begin{remark} The \eqref{ASOCQ} at $\bar x$ with respect to the direction $u=0$ reduces to the well-known Abadie constraint qualification $(ACQ)$; see \cite{Abadie}. As shown in  \cite{Abadie}, the $(ACQ)$ plays a fundamental role in establishing first-order  optimality conditions of the KKT form for nonlinear optimization problems.
\end{remark}
The following result ensures that the  \eqref{ASOCQ} holds at $\bar x$ with respect to $u$. 
 \begin{thm}\label{Pro-3.4} Let $\bar x\in X$ and  $u\in\R^n$. Suppose that the following system (in the unknown $w$)
 	\begin{eqnarray}
 	\langle \nabla g_i(\bar x), w\rangle+ \langle \zeta^{*i}, u\rangle<0, \ \ i\in I(\bar x; u),\label{equ:15}
 	\end{eqnarray}
 	has at least one solution, where
$$I(\bar x; u):=\{i\in I(\bar x): \langle\nabla g_i(\bar x), u\rangle=0 \}.$$ 
Then, the \eqref{ASOCQ} holds at $\bar x$ with respect to $u$.  
 \end{thm}
\begin{proof}
Let $\bar w\in\R^n$ be a solution of the system \eqref{equ:15} and fix any $v\in L^2(X; \bar x, u)$. We claim that $v\in T^2(X; \bar x, u)$.   Indeed, let $\{r_k\}$ and $\{t_j\}$ be any positive sequences converging to zero. We may assume that $r_k\in(0,1)$ for all $k\in\N$. For each $k\in\N$, put $v^k:=r_k\bar w+(1-r_k)v$. Clearly, $\displaystyle\lim_{k\to\infty}v^k=v$. Since
 $v\in L^2(X; \bar x, u)$, we have
 \begin{equation}\label{equa:6}
 G^2_i(u, v)=\left(\langle \nabla g_i(\bar x), u\rangle, \langle \nabla g_i(\bar x), v\rangle+\langle\zeta^{*i}, u\rangle\right)\leqq_{\rm lex} (0, 0), \ \ \ \forall i\in I(\bar x).
 \end{equation}
 This implies that $\langle \nabla g_i(\bar x), u\rangle\leq 0$ for all $i\in I(\bar x).$
  
For $k=1$, one has $v^1=r_1\bar w+(1-r_1)v$. We now show that the sequence $x^j:=\bar x+t_ju+\frac{1}{2}t_j^2v^1\in X$ for all $j$ large enough. To that end, we consider three cases as follows.
\\
 {\em Case 1.} $i\notin I(\bar x)$, i.e., $g_i(\bar x)<0$. Since $x^j\to \bar x$ as $j\to\infty$ and $g_i$ is continuous at $\bar x$, there exists $j_1\in\N$ such that $g_i(x^j)<0$ for all $j\geq j_1$.
\\
{\em Case 2.} $i\in I(\bar x)\setminus I^\prime(\bar x; u)$, i.e., $g_i(\bar x)=0$ and $\langle\nabla g_i(\bar x), u\rangle<0$. Since
$$\lim_{j\to\infty}\dfrac{g_i(x^j)}{t_j}=\lim_{j\to\infty}\dfrac{g_i(\bar x+t_ju+\frac12t^2_jv^{1})-g_i(\bar x)}{t_j}=\langle\nabla g_i(\bar x), u\rangle<0,$$ 
there is $j_2\in\N$ such that $g_i(x^j)<0$ for all $j\geq j_2$. 
\\
{\em Case 3.} $i\in I^\prime(\bar x; u)$, i.e., $g_i(\bar x)=0$ and $\langle\nabla g_i(\bar x), u\rangle=0$. It follows from \eqref{equa:6} that
$$\langle \nabla g_i(\bar x), v\rangle+\langle\zeta^{*i}, u\rangle\leq 0.$$
This and the fact that $\bar w$ is a solution of \eqref{equ:15} imply that
\begin{equation*}
 	\langle\nabla g_i(\bar x), v^1\rangle+\langle \zeta^{*i}, u\rangle=r_1[\langle\nabla g_i(\bar x), \bar w\rangle+\langle \zeta^{*i}, u\rangle]+(1-r_1)[\langle\nabla g_i(\bar x), v\rangle+\langle \zeta^{*i}, u\rangle]<0. \label{equa:2}
 \end{equation*}
By the right-hand side inequality of Theorem \ref{Taylor_formula} and an analysis similar to the one made in the proof of \eqref{equa-new-3} show  that there exist $\sigma_j\in (\bar x, \bar x+t_ju)$ and  $\zeta^{i_j}\in \partial^2 g_i(\sigma_j)(u)$ such that
 \begin{equation*} 
g_i(x^j)=g_i(x^j)-g_i(\bar x)-t_j\langle\nabla g_i(\bar x),u\rangle\leq 	\frac12t^2_j[\langle\nabla g_i(\bar x), v^{1}\rangle+\langle\zeta^{i_j},u\rangle],
 \end{equation*}
or, equivalently,
\begin{equation}\label{equ:18}
	\dfrac{g_i(x^j)}{\frac12t^2_j}\leq 	\langle\nabla g_i(\bar x), v^{1}\rangle+\langle\zeta^{i_j},u\rangle.
\end{equation}
Without any loss of generality, we may assume that $\zeta^{i_j}$ converges to some $\zeta^i\in \partial^2g_i(\bar x)(u)$ as $j\to\infty$. Taking the  limit superior in \eqref{equ:18} as $j\to\infty$, we obtain
\begin{align}
\limsup_{j\to\infty}\dfrac{g_i(x^j)}{\frac12t^2_j}&\leq \langle\nabla g_i(\bar x), v^{1}\rangle+\langle\zeta^{i},u\rangle\notag
\\
&\leq \langle\nabla g_i(\bar x), v^1\rangle+\langle \zeta^{*i}<0.\label{equa-new-4}
\end{align}
Hence, there exists $j_3\in\N$ such that $g_i(x^j)<0$ for all $j\geq j_3$.

Put $J_1:=\max\{j_1, j_2, j_3\}$, then $g_i(x^j)<0$ for all $j\geq J_1$. This implies that $x^{J_1}\in X$.  

Thus, by induction on $k$, we can construct a subsequence $x^{J_k}$  satisfying
 $$x^{J_k}=\bar x+t_{J_k}u+\frac12t^2_{J_k}v^{k}\in X,$$
 for all $k\in\N$. From this, $\displaystyle\lim_{k\to\infty}t_{J_k}=0$, and $\displaystyle\lim_{k\to\infty}v^k=v$ it follows that $v\in T^2(X; \bar x, u)$.  
 The proof is complete.   
\end{proof} 
 %================================================================
 \subsection{Second-Order Optimality Conditions}
 \label{Second_Order_condition}
  
 \begin{definition}[see \cite{Ehrgott}]\label{efficient_def}{\rm Let $\bar x\in X$. We say that:
\begin{enumerate}[\rm(i)]
	\item $\bar x$ is  {\em an efficient solution} (resp., {\em a weak efficient solution}) to problem \eqref{problem}  if there is no $x\in X$ satisfying  $f(x)\leq f(\bar x)$. (resp., $f(x)<f(\bar x)$).
	
	\item $\bar x$ is a {\em local efficient solution} (resp., {\em local weak efficient solution}) to problem \eqref{problem} if it is efficient solution (resp., weak efficient solution) in $U\cap X$ with some neighborhood $U$ of $\bar x$.
\end{enumerate}
 	}
 \end{definition}
 
 The following theorem gives a   first-order necessary optimality condition  for weak efficiency of \eqref{problem}.
 \begin{thm}[{see \cite[Theorem 3.1]{Huy162}}]\label{first_order-nec} If $\bar x\in X$ is a local weak efficient solution to problem \eqref{problem} and the $(ACQ)$ holds at $\bar x$, then the following system has no solution $u\in\R^n$:
 	\begin{eqnarray*}
 		\langle \nabla f_l(\bar x), u\rangle&< 0, \ \ \ &l\in L,  
 		\\
 		\langle \nabla g_i(\bar x), u\rangle&\leq 0, \ \ \ &i\in I (\bar x).  
 	\end{eqnarray*}
 \end{thm}

Let $\bar x\in X$ and $u\in\R^n$. We say that $u$ is a  {\em critical direction} at $\bar x$ if 
\begin{align*}
	\langle \nabla f_l(\bar x), u\rangle&\leq 0, \ \ \ \forall l\in L,
	\\
	\langle \nabla f_l(\bar x), u\rangle&= 0, \ \ \ \mbox{for at least one } \ \ l\in L,
	\\
	\langle \nabla g_i(\bar x), u\rangle&\leq 0, \ \ \ \forall i\in I(\bar x).
\end{align*}
The set of all critical direction of   \eqref{problem} at $\bar x$ is denoted by $\mathcal{C}(\bar x)$.  For each $u\in \mathcal{C}(\bar x)$, put
 \begin{equation*}
 	\mathcal{C}(\bar x, u):=\{w\in\R^n\,:\,\langle\nabla g_i(\bar x), w\rangle\leq 0,\ \ \ i\in I(\bar x; u)\}
 \end{equation*}
and 
\begin{equation*}
L(\bar x; u):=\{l\in L: \langle \nabla f_l(\bar x), u\rangle=0\}.
\end{equation*}
 The following theorem gives some second-order KKT necessary optimality conditions for a local weak efficient solution to  problem \eqref{problem}.
 \begin{thm}\label{theo2} Let $\bar x$ be a local  weak efficient solution  to problem \eqref{problem}. Suppose that the \eqref{ASOCQ} holds at $\bar x$ for any critical direction. Let $\bar u$ be a critical direction at $\bar x$. Then, there exist  $\lambda\in \R^m_+\setminus\{0\}$ and $\mu\in\R^m_+$  such that
 	\begin{align}
 		&\sum_{l=1}^m\lambda_l\nabla f_l(\bar x)+\sum_{i=1}^p\mu_i\nabla g_i(\bar x)=0, \label{equa:22}\\
 		&\sum_{l=1}^m\lambda_l\langle \xi^{*l}, \bar u\rangle+\sum_{i=1}^p\mu_i\langle \zeta^{*i}, \bar u\rangle\geq 0,\label{equa:23}
 		\\
 		&\lambda_l=0, l\notin L(\bar x; \bar u)  \label{equa:25}\\
 		& \mu_i=0,  i\notin I(\bar x; \bar u),\label{equa:24} 
 		\\
 		&{\sum_{l=1}^m\lambda_l\langle\nabla f_l(\bar x), w\rangle\geq 0, \ \ \ \forall w\in \mathcal{C}(\bar x, \bar u)\cap (\bar u)^{\bot}, }\label{equa:26}
 	\end{align} 
 	where 
 $$(\bar u)^{\bot}:=\{u\in\R^n:\langle \bar u, u\rangle=0\}.$$
 \end{thm}
\begin{proof}
The proof of the theorem follows some ideals of \cite[Theorem 3.2]{Huy162}. By assumptions, we first show that the following system 
 \begin{align}
 	F^2_l(u, v)&<_{\rm lex} (0, 0),\ \ \ l\in L, \label{equ:20}
 	\\
 	G^2_i(u, v)&\leqq_{\rm lex} (0, 0),\ \ \  i\in I(\bar x),\label{equ:22}
 \end{align}
 has no solution $(u, v)\in\R^n\times\R^n$. Suppose on the contrary that there exists $(u, v)\in\R^n\times\R^n$ satisfying \eqref{equ:20}--\eqref{equ:22}. This implies that $v\in L^2 (X; \bar x, u)$ and
 \begin{eqnarray*}
 	\langle \nabla f_l(\bar x), u\rangle&\leq 0, \ \ \ &l\in L,
 	\\
 	\langle \nabla g_i(\bar x), u\rangle&\leq 0, \ \ \ &i\in I (\bar x).
 \end{eqnarray*}
 Since the \eqref{ASOCQ} holds at $\bar x$ for any critical direction, so this condition holds at $\bar x$ for the direction $0$. This  means that the $(ACQ)$ is satisfied at $\bar x$. By Theorem \ref{first_order-nec},  $\langle \nabla f_l(\bar x), u\rangle=0$ for at least one $l\in L$. Hence, $u$ is a critical direction of problem \eqref{problem} at $\bar x$. 
 Since the \eqref{ASOCQ} holds at $\bar x$ for the critical direction $u$, we have that $v\in T^2(X; \bar x, u).$ This implies that there exist  sequences  $v^k$ converging to $v$ and   $t_k\downarrow 0$ such that
 $$x^k:=\bar x+t_ku+\frac12t_k^2v^k\in X,\ \   \forall k\in\N.$$ 
 Fix any $l\in L$. We consider two cases of $l$ as follows.
 \\
 {\em Case 1.} $l\in L(\bar x; u)$, i.e., $\langle \nabla f_l(\bar x), u\rangle=0$.  It follows from \eqref{equ:20} that
 \begin{equation*} 
 	\langle\nabla f_l(\bar x), v\rangle+\langle\xi^{*l}, u\rangle<0.
 \end{equation*}
 An analysis similar to the one made in the proof of \eqref{equa-new-4} shows that there exists $\xi^l\in \partial^2f_l(\bar x)(u)$ such that 
 \begin{align*}
 	\limsup_{k\to\infty}\dfrac{f_l(x^k)-f_l(\bar x)}{\frac12t^2_k}&\leq \langle\nabla f_l(\bar x), v\rangle+\langle\xi^l, u\rangle \notag
 	\\
 	&\leq \langle\nabla f_l(\bar x), v\rangle+\langle\xi^{*l}, u\rangle<0.\label{equ:23}
 \end{align*}
This  implies that there exists $k_1\in\N$ such that $f_l(x^k)-f_l(\bar x)<0$ for all $k\geq k_1$.
\\
{\em Case 2.} $l\in L\setminus L(\bar x; u)$, i.e., $\langle \nabla f_l(\bar x), u\rangle<0$. Then
 $$\lim_{k\to\infty}\dfrac{f_l(x^k)-f_l(\bar x)}{t_k}=\langle \nabla f_l(\bar x), u\rangle<0.$$
Hence, there exists $k_2\in\N$ such that $f_l(x^k)-f_l(\bar x)<0$ for all $k\geq k_2$.

Put $k_0:=\max\{k_1, k_2\}$. Then we see that
$f_l(x^k)-f_l(\bar x)<0$ 
 for all $l\in L$ and $k\geq k_0$, which contradicts the fact that $\bar x$ is a local weak efficient solution of \eqref{problem}. 
 
We now fix any $\bar u\in \mathcal{C}(\bar x)$. Then, the above arguments show  that the following system 
 \begin{align*}
 	F^2_l(\bar u, v)&<_{\rm lex} (0, 0),\ \ \ l\in L,  
 	\\
 	G^2_i(\bar u, v)&\leqq_{\rm lex} (0, 0),\ \ \  i\in I(\bar x), 
 \end{align*}
 has no solution $v\in \mathbb{R}^n$. This means that the following system
 \begin{eqnarray*}
 	\langle \nabla f_l(\bar x), v\rangle+\langle \xi^{*i}, \bar x\rangle&< 0,\ \ \ \ &l\in L(\bar x; \bar u), \label{equ:24}
 	\\
 	\langle \nabla g_i(\bar x), v\rangle+ \langle \zeta^{*i}, \bar x\rangle&\leq 0,\ \ \ \ &i\in I(\bar x; \bar u),\label{equ:26}
 \end{eqnarray*}   
 has no solution $v\in\R^n$. By the Motzkin theorem of the alternative \cite[p. 28]{Mangasarian69}, there exist $\lambda\in \R^m_+\setminus\{0\}$ and $\mu\in\R^p_+$  such that
 \begin{align*}
 	&\sum_{l=1}^m\lambda_l\nabla f_l(\bar x)+\sum_{i=1}^p\mu_i\nabla g_i(\bar x)=0,  \\
 	&\sum_{l=1}^m\lambda_l\langle \xi^{*l}, \bar u\rangle+\sum_{i=1}^p\mu_i\langle \zeta^{*i}, \bar u\rangle\geq 0, 
 	\\
 	&\lambda_l=0, l\notin L(\bar x; \bar u)
 	\\
 	& \mu_i=0,  i\notin I(\bar x; \bar u).
 \end{align*} 
 We  now see that 
 $$\left\langle \sum_{l=1}^m\lambda_l\nabla f_l(\bar x)+\sum_{i=1}^p\mu_i\nabla g_i(\bar x), w\right\rangle=0$$
 for all $w\in \R^n$. Hence, if $w\in \mathcal{C}(\bar x; \bar u)\cap (\bar u)^{\bot}$, then we have that
 $$\sum_{l=1}^m\lambda_l \langle \nabla f_l(\bar x), w\rangle=-\sum_{i=1}^p\mu_i \langle\nabla g_i(\bar x), w\rangle=-\sum_{i\in I(\bar x, \bar u)}\mu_i \langle\nabla g_i(\bar x), w\rangle\geq 0.$$

 Since \eqref{equa:25} and $w\in \mathcal{C}(\bar x; \bar u)$, we have
 $$\sum_{l=1}^m\lambda_l \langle \nabla f_l(\bar x), w\rangle=-\sum_{i\in I(\bar x; \bar u)}\mu_i \langle\nabla g_i(\bar x), w\rangle\geq 0.$$
 Thus \eqref{equa:26} holds true. The proof is complete. 
\end{proof} 
 
\begin{remark}
Condition \eqref{equa:23} can be stated as follows:
\begin{equation*}
\sum_{l=1}^m\lambda_l\max\{\langle \xi^{l}, \bar u\rangle: \xi^l\in\partial^2f_l(\bar x)(u)\}+\sum_{i=1}^p\mu_i\max\{\langle \zeta^{i}, \bar u\rangle:\zeta^i\in\partial^2g_i(\bar x)(u)\}\geq 0.
\end{equation*}
Since the limiting second-order subdifferential is strictly smaller than the second-order symmetric subdifferential, our result Theorem \ref{theo2} improves the corresponding result  \cite[Theorem 3.2]{Huy162}.
\end{remark} 
 
 The vector $(\lambda, \mu)\in(\R^m_+\setminus\{0\})\times\R^p_+$ satisfying condition \eqref{equa:22}--\eqref{equa:26} is called {\em  a pair of weak second-order KKT multipliers}. If we can choose $(\lambda, \mu)\in(\R^m_+\setminus\{0\})\times\R^p_+$ such that $\lambda_l>0$ for all $l\in L$, then $(\lambda, \mu)$ is called {\em  a pair of strong second-order KKT multipliers}. 
    
The following theorem gives some  sufficient conditions of the strong second-order KKT form  for a local efficient solution of problem \eqref{problem}.
 \begin{thm} \label{KKT_sufficient}
 	Let $\bar x\in X$. Suppose that the $(ACQ)$ holds at $\bar x$ and for each $u\in \mathcal{C}(\bar x)\setminus\{0\}$ there exist  $\lambda\in\R^m_+$ and $\mu\in \R^p_+$ such that
 	\begin{align}
 		&\sum_{l=1}^m\lambda_l\nabla f_l(\bar x)+\sum_{i=1}^p\mu_i\nabla g_i(\bar x)=0,\label{equ:34}
 		\\
 		&\sum_{l=1}^m\lambda_l\langle \xi_*^l, u\rangle+\sum_{i=1}^p\mu_i\langle \zeta_*^i, u\rangle> 0,\label{equ:35}
 		\\
 		& \lambda_l> 0, \ \ \forall l\in L,\label{equ:37}
 		\\
 		& \mu_i=0,\ \ \ i\notin I(\bar x; u),\label{equ:36}
 		 		\\
 		&\sum_{l=1}^m\lambda_l\langle\nabla f_l(\bar x), w\rangle>0, \ \ \ \forall w\in \mathcal{C}(\bar x, u)\cap u^{\bot} \setminus\{0\}, \label{equa:17}
 	\end{align}
 then $\bar x$ is a local efficient solution of \eqref{problem}.     
 \end{thm}
\begin{proof}
The proof of the theorem follows some ideals of \cite[Theorem 3.6]{Huy162}. Suppose on the contrary that $\bar x$ is not a local efficient solution of \eqref{problem}. Then, there exists a sequence $x^k\in X$ that converges to $\bar x$ and satisfies
 \begin{equation}\label{equ:38}
 	f(x^k)\leq f(\bar x), \ \ \ \forall k\in\N.
 \end{equation}
This implies that $x^k\neq \bar x$ for all $k\in\N$. Hence, for each $k\in\N$, put $t_k:=\|x^k-\bar x\|$. Then $t_k\downarrow 0$ as $k\to\infty$. Let $ u^k:=\frac1{t_k}(x^k-\bar x).$ Then, $\|u_k\|=1$. Without any loss of generality, we may assume that $\{u^k\}$ converges to some $u\in\R^n$ with $\|u\|=1$. By the mean value theorem for differentiable functions and \eqref{equ:38}, we have
 \begin{equation*}
 0\geq 	f_l(x^k)-f_l(\bar x)=t_k\langle \nabla f_l(\bar x), u^k\rangle +o(t_k),\ \ \ \forall k\in\N, l\in L. \label{critical:4}
 \end{equation*}
This implies that 
$$\langle \nabla f_l(\bar x), u \rangle=\lim_{k\to\infty}\langle \nabla f_l(\bar x), u^k \rangle=\lim_{k\to\infty}\dfrac{f_l(x^k)-f_l(\bar x)}{t_k}\leq 0, \ \ \forall l\in L.$$
Similarly, since $g_i(x^k)=g_i(x^k)-g_i(\bar x)\leq 0$ when $i\in I(\bar x)$, we obtain $$\langle \nabla g_i(\bar x), u \rangle\leq 0, \ \ \forall i\in I(\bar x).$$ 
By the  $(ACQ)$ and  Theorem \ref{first_order-nec}, there exists at least one  $l\in L$ such that $\langle \nabla f_l(\bar x), u \rangle=0$. This implies that  $u\in\mathcal{C}(\bar x)$ and $\|u\|=1$.

By assumptions, there exist  $\lambda\in\R^m_+$ and $\mu\in \R^p_+$ satisfying \eqref{equ:34}--\eqref{equa:17}.  It is easy to see from \eqref{equ:37} that $\langle \nabla f_l(\bar x), u \rangle= 0$ for all $l\in L$. Thus, we have 
 \begin{align*}
 	f_l(x^k)-f_l(\bar x)=[f_l(\bar x+t_ku^k)-f_l(\bar x+t_ku)]+[f_l(\bar x+t_ku)-f_l(\bar x)-t_k\langle\nabla f_l(\bar x), u\rangle].
 \end{align*}
 It follows from the differentiability of $f_l$ that there exists $\theta^{l_k}\in(\bar x+t_ku, x^k)$ satisfying 
 $$f_l(x^k)-f_l(\bar x+t_ku)=t_k\langle \nabla f_l(\theta^{l_k}), u^k-u\rangle.$$
By Theorem \ref{Taylor_formula} and  an analysis similar to the one made in the proof of \eqref{equa-new-5}, there exist $\gamma^{l_k}\in(\bar x, \bar x+t_ku)$ and  $\xi^{l_k}\in\partial^2f_l(\gamma^{l_k})(u)$ satisfying
 $$f_l(\bar x+t_ku)-f_l(\bar x) -t_k\langle\nabla f_l(\bar x), u\rangle \geq \frac12t^2_k\langle \xi^{l_k}, u\rangle.$$
Hence
 \begin{equation*} 
 0\geq f_l(x^k)-f_l(\bar x)\geq t_k\langle \nabla f_l(\theta^{l_k}), u^k-u\rangle+ \frac12t^2_k\langle \xi^{l_k}, u\rangle,
 \end{equation*}
or, equivalently,
\begin{equation}\label{equa:9}
	\langle \nabla f_l(\theta^{l_k}), u^k-u\rangle+ \frac12t_k\langle \xi^{l_k}, u\rangle\leq 0.
\end{equation}
 Similarly, for each $k\in \N$ and $i\in I(\bar x; u)$, there are $\tau^{i_k}\in (\bar x+t_ku, x^k)$, $\sigma^{i_k} \in (\bar x, \bar x+t_ku)$ and $\zeta^{i_k}\in \partial^2g_i (\sigma^{i_k})(u)$ such that
 \begin{equation*} 
 0\geq g_i(x^k)-g_i(\bar x)\geq t_k\langle\nabla g_i(\tau^{i_k}), u^k-u\rangle+\frac12t^2_k\langle \zeta^{i_k}, u\rangle,
 \end{equation*}
or, equivalently,
\begin{equation}\label{equa:10}
\langle\nabla g_i(\tau^{i_k}), u^k-u\rangle+\frac12t_k\langle \zeta^{i_k}, u\rangle\leq 0.
\end{equation}
 By Proposition \ref{property}, without loss any of generality, we may assume that $\xi^{l_k}$ (resp. $\zeta^{i_k}$) converges to $\xi^l\in \partial^2f_k(\bar x)(u)$ (resp. $\zeta^i\in \partial^2g_i(\bar x)(u)$). 
Combining \eqref{equa:9}, \eqref{equa:10}, and \eqref{equ:36}, we obtain
\begin{align}
 \sum_{l=1}^m\lambda_l\left[\langle \nabla f_l(\theta^{l_k}), u^k-u\rangle+ \frac12t_k\langle \xi^{l_k}, u\rangle\right]+\sum_{i=1}^p\mu_i\left[\langle\nabla g_i(\sigma^{i_k}), u^k-u\rangle+\frac12t_k\langle \zeta^{i_k}, u\rangle\right]\leq 0.\label{equa:11}
 \end{align}
 For  each $k\in\N$, put $s_k:=\|u^k-u\|$ and $w^k:=\frac{u^k-u}{s_k}$. Then, \eqref{equa:11} is equivalent to
\begin{align}
\sum_{l=1}^m\lambda_l\left[s_k\langle \nabla f_l(\theta^{l_k}), w^k\rangle+ \frac12t_k\langle \xi^{l_k}, u\rangle\right]+\sum_{i=1}^p\mu_i\left[s_k\langle\nabla g_i(\sigma^{i_k}), w^k\rangle+\frac12t_k\langle \zeta^{i_k}, u\rangle\right]\leq 0.\label{equa:12}
\end{align} 
Since $\|w^k\|=1$ for all $k\in\N$, without any loss of generality, we may assume that $w^k$ converges to some $w\in\R^n$ with $\|w\|=1$. 
By passing to subsequences if necessary we  may  consider three cases of sequences $t_k$ and $s_k$ as follows.
 \\
{\em Case 1.} $\displaystyle\lim_{k\to\infty}\frac{s_k}{t_k}=0$. Dividing the two sides of \eqref{equa:12} by $\frac 12t_k$ and then taking to the limit when $k\to\infty$ we obtain
$$\sum_{l=1}^m\lambda_l\langle \xi^l, u\rangle+\sum_{i=1}^m\mu_i\langle \zeta^i, u\rangle\leq 0.$$
Thus
 $$\sum_{l=1}^m\lambda_l\langle \xi_*^l, u\rangle+\sum_{i=1}^m\mu_i\langle \zeta_*^i, u\rangle\leq \sum_{l=1}^m\lambda_l\langle \xi^l, u\rangle+\sum_{i=1}^m\mu_i\langle \zeta^i, u\rangle\leq 0,$$
 contrary to \eqref{equ:35}.
 \\
 {\em Case 2.} $\displaystyle\lim_{k\to\infty}\frac{s_k}{t_k}=r>0$. Dividing the two sides of \eqref{equa:12} by $\frac 12t_k$ and then taking to the limit when $k\to\infty$ we obtain
 $$\sum_{l=1}^m\lambda_l[r\langle\nabla f_l(\bar x), w\rangle+\langle \xi^l, u\rangle]+\sum_{i=1}^m\mu_i [r\langle\nabla g_i(\bar x), w\rangle +\langle \zeta^i u\rangle]\leq 0.$$
This and \eqref{equ:34} imply that $$\sum_{l=1}^m\lambda_l\langle \xi^l, u\rangle+\sum_{i=1}^m\mu_i\langle \zeta^i, u\rangle\leq 0,$$ 
again contrary to \eqref{equ:35}.
 \\
{\em Case 3.} $\displaystyle\lim_{k\to\infty}\frac{s_k}{t_k}=+\infty$, or, equivalently,  $\displaystyle\lim_{k\to\infty}\frac{t_k}{s_k}=0$.  For each $k\in\N$, one has 
$$x^k=\bar x+t_ku^k=\bar x+t_ku+t_ks_kw^k.$$ 
Hence,
$$f_l(x^k)-f_l(\bar x)=[f_l(x^k)-f_l(\bar x+t_ku)]+[f_l(\bar x+t_ku)-f_l(\bar x)-t_k\langle\nabla f_l(\bar x), u\rangle]$$
 for all $l\in L$ and $k\in\N$. By  an analysis similar to the one made in the proof of \eqref{equa:9} we can find $y^{l_k}\in (\bar x+t_ku, x^k)$, $\gamma^{l_k}\in(\bar x, \bar x+t_ku)$ and  $\xi^{l_k}\in\partial^2f_l(\gamma^{l_k})(u)$ such that
 \begin{equation*}\label{equa:18}
 0\geq 	f_l(x^k)-f_l(\bar x)\geq t_ks_k\langle \nabla f_l(y^{l_k}), w^k\rangle+ \frac12t^2_k\langle \xi^{l_k}, u\rangle.
 \end{equation*}
Hence,
 \begin{equation}\label{equa:20}
 	\langle \nabla f_l(y^{l_k}), w^k\rangle+ \frac12\frac{t_k}{s_k}\langle \xi^{l_k}, u\rangle\leq 0.
 \end{equation}
 Letting $k\to\infty$ in \eqref{equa:20} we obtain $\langle \nabla f_l(\bar x), w\rangle\leq 0 $  for all $l\in L$ and so
 \begin{equation*}\label{equa:21}
 \sum_{i=1}^l\lambda_i\langle\nabla f_i(\bar x), w\rangle\leq 0. 
 \end{equation*}
We now show that $w\in K(\bar x, u)\cap u^{\bot} \setminus\{0\}$ and arrive at a contradiction. Indeed, since $u^k=u+r_kw^k\to u$, $w^k\to w$ as $k\to\infty$, and $u^k=u+r_kw^k\in \mathbb{S}^n$ for all $k\in \N$, we have $w\in T(\mathbb{S}^n; u)=u^{\bot}$.  Hence, $w\in K(\bar x, u)\cap u^{\bot} \setminus\{0\}$. The proof is complete.
\end{proof}  
\begin{remark}
	Condition \eqref{equ:35} can be stated as follows:
	\begin{equation*}
		\sum_{l=1}^m\lambda_l\min\{\langle \xi^{l}, \bar u\rangle: \xi^l\in\partial^2f_l(\bar x)(u)\}+\sum_{i=1}^p\mu_i\min\{\langle \zeta^{i}, \bar u\rangle:\zeta^i\in\partial^2g_i(\bar x)(u)\}> 0.
	\end{equation*}
Since the limiting second-order subdifferential is strictly smaller than the second-order symmetric one, our result Theorem \ref{KKT_sufficient} improves the corresponding one  \cite[Theorem 3.6]{Huy162}.
\end{remark}

\section{Conclusion} \label{Discussion}
By using the limiting second-order Taylor formula in the form of inequalities for $C^{1,1}$ functions, we obtain second-order KKT necessary optimality conditions for efficiency (Theorem \ref{theo2}) and a strong  second-order KKT sufficient optimality condition (Theorem \ref{KKT_sufficient}) for local efficient solutions of $C^{1,1}$ vector optimization problems with inequality constraints. These results improve and generalize the corresponding of Huy et al. \cite[Theorems 3.2 and 3.6]{Huy162} and of Feng and Li \cite{Feng_Li_2020}. By a similar way, we can also drive results that improve the corresponding ones of Huy et al. \cite[Theorems 3.3--3.5]{Huy162} and of Tuyen et al. \cite[Theorem 4.5]{Tuyen-Huy-Kim}.

\section*{Statements and Declarations} 
The author declares that he has no conflict of interest, and the manuscript has no associated data.

\section*{Acknowledgments.}
This research is funded by Hanoi Pedagogical University 2.

\end{document}